\documentclass[11pt]{amsart}
\usepackage{amsmath,amssymb}
\usepackage{graphicx, float}

\usepackage{fullpage}
\usepackage{enumerate, enumitem}
\usepackage{bbm}
\usepackage{pgf,tikz}
\usetikzlibrary{arrows}
\usetikzlibrary{positioning}
\usetikzlibrary{patterns}

\def\COMMENT#1{}
\let\COMMENT=\footnote
\def\TASK#1{}

\newdimen\margin   
\def\textno#1&#2\par{%
    \margin=\hsize
    \advance\margin by -4\parindent
           \setbox1=\hbox{\sl#1}%
    \ifdim\wd1 < \margin
       $$\box1\eqno#2$$%
    \else
       \bigbreak
       \hbox to \hsize{\indent$\vcenter{\advance\hsize by -3\parindent
       \sl\noindent#1}\hfil#2$}%
       \bigbreak
    \fi}

\newtheorem{thm}{Theorem}[section]
\newtheorem{define}[thm]{Definition}
\newtheorem{example}[thm]{Example}
\newtheorem{lem}[thm]{Lemma}
\newtheorem{claim}[thm]{Claim}

\newtheorem{prop}[thm]{Proposition}

\newtheorem{question}[thm]{Question}

\newtheorem*{thm*}{Theorem}
\newtheorem*{define*}{Definition}
\newtheorem*{examp*}{Example}
\newtheorem*{lem*}{Lemma}
\newtheorem*{claim*}{Claim}
\newtheorem*{fact*}{Fact}
\newtheorem*{col*}{Corollary}
\newtheorem*{conj*}{Conjecture}

\begin{document}

\title{A note on colour-bias perfect matchings in hypergraphs}

\author{J\'ozsef Balogh,
 Andrew Treglown and Camila Z\'arate-Guer\'en}

\thanks{JB:  University of Illinois Urbana-Champaign, 1409 W. Green Street, Urbana IL 61801, United States.  Research 
\indent  supported in part by NSF grants DMS-1764123 and RTG DMS-1937241, FRG DMS-2152488, the Arnold O. 
\indent Beckman Research Award (UIUC Campus Research Board RB 24012), the Langan Scholar Fund (UIUC).\\
\indent AT: University of Birmingham, United Kingdom, {\tt a.c.treglown@bham.ac.uk}. Research supported by EPSRC \indent grant EP/V002279/1.\\
\indent CZG: University of Birmingham, United Kingdom, {\tt ciz230@student.bham.ac.uk}. Research supported by EPSRC}

\date{}
\begin{abstract}
A result of Balogh, Csaba, Jing and Pluh\'ar yields the minimum degree threshold that ensures a $2$-coloured graph  contains a perfect matching of significant colour-bias (i.e., a perfect matching that contains significantly more than half of its edges in one colour).
In this note we prove an analogous result for perfect matchings in $k$-uniform hypergraphs. More precisely, for each $2\leq \ell <k$ and $r\geq 2$ we determine the minimum $\ell$-degree threshold for forcing a perfect matching of significant colour-bias in  an $r$-coloured $k$-uniform hypergraph.
\end{abstract}
\maketitle

\section{Introduction}
A \emph{perfect matching} in a hypergraph $H$ is a collection of vertex-disjoint edges of $H$ which covers the vertex set $V(H)$ of $H$. 
In recent decades there has been significant interest in the problem of establishing \emph{minimum degree} conditions that force a perfect matching in a $k$-uniform
hypergraph. More precisely, given a $k$-uniform hypergraph $H$ and an $\ell$-element vertex set $S\subseteq V(H)$ (where $ \ell \in [k-1]$) we define
$d_H (S)$ to be the number of edges containing $S$. The \emph{minimum $\ell$-degree $\delta _{\ell} (H)$} 
of $H$ is the minimum of $d_H (S)$ over all $\ell$-element sets of vertices in $H$. We 
refer to  $\delta _1 (H)$ as the \emph{minimum vertex degree} of $H$ and  $\delta _{k-1}
(H)$ as the \emph{minimum codegree} of $H$.

Suppose that $\ell, k, n \in \mathbb N$ such that  $  \ell \le k-1$ and $k$ divides $n$.
Let $m_{\ell}(k, n)$ denote the smallest integer $m$ such that every $k$-uniform hypergraph $H$ on $n$ vertices with $\delta _{\ell}(H)\ge m$ contains a perfect matching. 

A simple consequence of Dirac's theorem is that $m_{1}(2, n)=n/2$ for all even $n\in \mathbb N$.
Improving earlier asymptotically exact bounds given in~\cite{ko1, rrs2},
R\"odl, Ruci\'nski and Szemer\'edi~\cite{rrs} determined the minimum codegree threshold for perfect matchings in $k$-uniform hypergraphs. That is, they showed that if 
$n \in \mathbb N$ 
is sufficiently large, then $m_{k-1}(k, n)=n/2 - k + C$, where $C\in \{3/2, 2, 5/2, 3\}$ depends on the values of $n$ and $k$.

The value of $m_{\ell}(k, n)$ is known for various pairs $(k,\ell)$ when $n$ is sufficiently large. For example, after an earlier asymptotic result of Pikhurko~\cite{pik}, Treglown and Zhao~\cite{zhao2} determined the value of $m_{\ell}(k, n)$ for $\ell \geq k/2$ and
$n$  sufficiently large. However, the minimum vertex degree case of the problem is wide open in general, and the only cases where the asymptotic or exact value of $m_{1}(k, n)$
is known is when $k=2,3,4,5$. 
See, e.g.,~\cite{rrsurvey, zsurvey} for  discussions on further results in the area.

Given any $1\leq \ell <k$ it is known that
\begin{align}\label{mlower}
m_{\ell}(k, n) \geq 
\max 
\left \{ 
\frac{1}{2}-o(1), 1- \left ( \frac{k-1}{k}\right )^{k-\ell} -o(1) \right\}\binom{n}{k-\ell}.
\end{align}
See, e.g., the introduction of~\cite{zhao3} for the two families of hypergraphs that demonstrate  (\ref{mlower}). 
It is widely believed that the inequality in (\ref{mlower}) is asymptotically sharp for all choices of $k, \ell$, see~\cite{hps, kosurvey}. Moreover, Treglown and Zhao~\cite{zhao3} gave a  conjecture on the exact value of $m_{\ell}(k, n)$ for sufficiently large $n \in k\mathbb N$. 

\smallskip

The aim of this paper is to  study  the \emph{colour-bias} version of this problem. The topic of colour-bias structures in graphs was first raised by Erd\H{o}s in the 1960s (see~\cite{e2, e1}). 
Sparked by  work of Balogh, Csaba, Jing and Pluh\'ar~\cite{BCsJP}, 
 there has been renewed interest in the topic, particularly in establishing minimum degree conditions that
force a colour-bias copy of a graph $F$.
More precisely, if a graph $G$ contains a copy of $F$, then however  
the edges of $G$ are $2$-coloured, one can clearly ensure that $G$ contains a copy of $F$ with at least $e(F)/2$ edges
of the same colour. The question then is how large does the minimum degree $\delta (G)$ of $G$ need
to be to guarantee that $G$ contains a copy of $F$ with significantly more than $e(F)/2$ edges
of the same colour, no matter how one $2$-colours
the edges of $G$? The following result resolves this problem in the case when $F$ is a Hamilton cycle.
\begin{thm}[Balogh, Csaba, Jing and Pluh\'ar~\cite{BCsJP}]\label{old}
Let  $0<c<1/4$  and $n\in \mathbb N$ be sufficiently large. If $G$ is an $n$-vertex graph with $$\delta(G)\geq(3/4+c)n,$$
 then given any $2$-colouring of $E(G)$ 
there is a Hamilton cycle in $G$ with at least 
$ n/2+cn/32$ edges of the same colour.
Moreover, if $n\in 4 \mathbb N$, there is an $n$-vertex graph $G'$ with $\delta(G')=3n/4$ and
a $2$-colouring of $E(G')$ for which every Hamilton cycle in $G'$ has precisely $n/2$ edges in each colour.
\end{thm}
Note that Theorem~\ref{old} shows that the minimum degree threshold for forcing a \emph{colour-bias} Hamilton cycle in a graph is significantly higher than the threshold for just forcing a Hamilton cycle. Indeed, Dirac's theorem tells us that any $n$-vertex graph $G$ with $\delta (G)\geq n/2$  contains a Hamilton cycle.

Since a Hamilton cycle on an even number of vertices is the union of two perfect matchings,
Theorem~\ref{old}  implies the following result.
\begin{thm}[Balogh, Csaba, Jing and Pluh\'ar~\cite{BCsJP}]\label{cor1}
Let  $0<c<1/4$  and $n\in 2 \mathbb N$ be sufficiently large. If $G$ is an $n$-vertex graph with $$\delta(G)\geq(3/4+c)n,$$
 then given any $2$-colouring of $E(G)$ 
there is a perfect matching in $G$ with at least 
${n}/{4}+{cn}/{64}$ edges of the same colour.
Moreover, if $n\in 4 \mathbb N$, there is an $n$-vertex graph $G'$ with $\delta(G')=3n/4$ and
a $2$-colouring of $E(G')$ for which every perfect matching in $G'$ has precisely $n/4$ edges in each colour.
\end{thm}
Let $n \in 4 \mathbb N$. We define the graph $G'$ in Theorem~\ref{cor1} as follows:
$V(G')$ consists of the disjoint union of two vertex classes $A$ and $B$ of sizes $n/4$ and $3n/4$,
respectively; $E(G')$ contains all possible red edges whose endpoints are both in $B$ and
all possible blue edges with one endpoint in $A$ and one endpoint in $B$. 
Thus, $\delta(G')=3n/4$ and every perfect matching in $G'$ has precisely $n/4$ edges in each colour.

Since~\cite{BCsJP} appeared, a number of analogues of Theorem~\ref{old} have been established for other types of spanning structures.
Given graphs $G$ and $F$, an \emph{$F$-factor} in $G$ is a collection of vertex-disjoint copies 
of $F$ in $G$ that together cover $V(G)$.
In~\cite{cpc}, the minimum degree threshold for forcing a colour-bias $K_r$-factor was determined.\footnote{Recall $K_r$ denotes the complete graph on $r$ vertices.} More recently, this result was extended to $F$-factors for every fixed graph $F$; see~\cite{brad2}.
For $k \geq 2$, the minimum degree threshold for forcing a colour-bias $k$th power of a Hamilton cycle in a graph was established in~\cite{brad}.

Other variants of the problem have also been studied. In~\cite{fhlt, gkm} an $r$-colour version
of Theorem~\ref{old} was proven: in this setting now one $r$-colours $E(G)$ and seeks a Hamilton cycle with significantly more than $n/r$ edges of the same colour.
Colour-bias problems have also been considered for random graphs~\cite{newkri}.
Recently, Mansilla Brito~\cite{mb} gave a minimum codegree result for forcing 
a colour-bias copy of a tight Hamilton cycle in a $3$-uniform hypergraph.
We remark that all of these colour-bias results can be phrased in the equivalent language of \emph{discrepancy}; see, e.g.,~\cite{BCsJP, cpc, brad, brad2, gkm}.

\smallskip

Our main result determines the minimum $\ell$-degree threshold for forcing a colour-bias perfect matching in a $k$-uniform hypergraph for {all} $\ell \geq 2$ and $k \geq 3$.  To state our result we need the following definitions.
Given  integers $1\leq \ell <k$, let 
$\mathcal C_{k,\ell}$ be the set of all $c>0$ such that 
$m_{\ell} (k,n)\leq c \binom{n}{k-\ell}$ for all sufficiently large $n\in k\mathbb N$. Set $c_{k,\ell}$ to be the infimum of $\mathcal C_{k,\ell}$.
In particular, note that the general conjecture on the asymptotic value of $m_{\ell} (k,n)$
equivalently states that 
\begin{align*}
c_{k,\ell} = 
\max 
\left \{ 
\frac{1}{2}, 1- \left ( \frac{k-1}{k}\right )^{k-\ell}  \right\}.
\end{align*}

\begin{thm}\label{thmmain2}
Let $k, \ell, r \in \mathbb N$ where $2\leq \ell <k$ and $r \geq 2$. Given any $\eta >0$ where $c_{k,\ell}+\eta < 1$, there exists an $n_0 \in \mathbb N$
such that the following holds.
Let $H$ be a $k$-uniform hypergraph on $n \geq n_0$ vertices, where $n \in k \mathbb N$.
If 
$$ \delta _{\ell} (H) \geq  (c_{k,\ell}+\eta) \binom{n}{k-\ell},$$
then  given any $r$-colouring of $E(H)$ there is a perfect matching in $H$ with at least
$\frac{n}{rk}+\frac{\eta n}{8r(r-1)k^k(k^2+k)} $ edges of the same colour.
\end{thm}

We remark that Theorem~\ref{thmmain2} holds even in the cases in which we do not know the value of $c_{k,\ell}$. By definition of $c_{k,\ell}$,
 the minimum $\ell$-degree condition in Theorem~\ref{thmmain2} is essentially best possible. Indeed, for $c<c_{k,\ell}$, a minimum $\ell$-degree condition of
 $ \delta _{\ell} (H) \geq  c \binom{n}{k-\ell} $ does not even guarantee
a perfect matching, let alone one of significant colour-bias.
So in this sense the colour-bias and `standard' versions of the problem are aligned when $\ell \geq 2$. 

 In contrast, the same phenomenon does not occur for the minimum vertex degree version of the problem. Indeed, 
 Theorem~\ref{cor1}  tells
us that the minimum degree threshold for a colour-bias perfect matching in a \emph{graph} 
is different to the  minimum degree threshold for a  perfect matching in a graph. Furthermore, in Section~\ref{sec:conc} we describe a similar phenomenon in the $3$-uniform hypergraph setting.


\smallskip

{\noindent \bf Remark.} Whilst finalising a manuscript 
that gave the proof of Theorem~\ref{thmmain2} in the case when $\ell=k-1$ and $r=2$,
we learnt of simultaneous and independent work of Gishboliner, Glock and Sgueglia~\cite{ggs}. In~\cite{ggs} they 
determine the \emph{minimum codegree threshold} for forcing a tight Hamilton cycle of significant colour-bias in an $r$-coloured $k$-uniform hypergraph (where $r \geq 2$ and $k \geq3$). As an immediate consequence of their result they also
establish the corresponding \emph{minimum codegree threshold} for perfect matchings.

We therefore decided to seek a generalisation of our minimum codegree result to other degree conditions, i.e., Theorem~\ref{thmmain2}. In doing so, we found an argument much cleaner than our original approach.

\smallskip

{\noindent \bf Notation.}
Let $H$ be a hypergraph. 
The \textit{neighbourhood $N_H(X)$ of a set $X\subseteq V(H)$} is the family of  sets $S\subseteq V(H)\setminus X$ such that $S\cup X\in E(H)$.
If $X=\{x\}$ we define $N_H(x):=N_H(X)$.
Given a vertex $x \in V(H)$ and set $Y \subseteq V(H)$ we sometimes write $x Y$ or $Yx$ to denote $\{x\}\cup Y$.
Given a colouring $c$ of $E(H)$, we call an edge $e \in E(H)$ a \emph{$C$-edge} if $e$ is coloured $C$ in $c$.
Given a set $X\subseteq V(H)$, we write $H[X]$ for the \emph{induced subhypergraph of $H$ with vertex set $X$}. We define $H\setminus X:=H[V(H)\setminus X]$.

Given a hypergraph $F$ with an $r$-colouring $c: E(F)\to \{C_1, \dots, C_r\}$, its \emph{colour profile} is $(x_1, \dots,x_r)$ where $x_i$ is the number of $C_i$-edges in $F$ for each $i\in[r]$. Two colour profiles $(x_1, \dots, x_r)$, $(y_1, \dots, y_r)$ are said to be \emph{different with respect to the colour $C_i$} if $x_i\neq y_i$.




\section{Preliminaries and useful results}\label{sec:pre}
\subsection{Proof overview and key definitions}
Throughout this section, we will suppose that $H$ is a $k$-uniform hypergraph on $n$ vertices with an $r$-colouring $c: E(H)\to \{C_1, \dots, C_r\}$. 

Our general strategy for the proof of Theorem~\ref{thmmain2} is as follows. 
Our aim is to find certain \emph{gadgets} inside of $H$. A gadget  is just a subhypergraph of $H$ with some given structure. 
 A gadget $G$ is \emph{good} if $G$ contains two  perfect matchings that have different colour profiles with respect to the $r$-colouring $c$.

For a  certain well chosen $t \in \mathbb N$, we will prove that there are
 $t$  vertex-disjoint good gadgets $G_1,\dots, G_t$ in $H$ 
 and a $j \in [r]$ so that, for each good gadget $G_i$, the two perfect matchings $M_i$ and $M'_i$ in $G_i$  have colour profiles that are different with respect to the colour $C_j$.

We will then be able to easily find a perfect matching in $H$ of significant colour-bias. Indeed, removing the vertices of $G_1,\dots,G_t$ from $H$ will result in a
$k$-uniform hypergraph $H'$ that  contains a perfect matching $M$. The flexibility of the good gadgets then allows us to extend $M$ into a perfect matching in $H$ with significant colour-bias, whatever the colour profile of $M$ is.

We next state the definitions required to formally introduce the notion of a good gadget.

\begin{define}
Let $u,v\in V(H)$ be distinct and  $T\in N_H(u)\cap N_H(v)$. We say $uTv$ is
\begin{itemize}
    \item $\bf S$ if $c(T\cup\{u\})=c(T\cup\{v\})$; or
    \item $\mathbf{C_iC_j}$ if $c(T\cup\{u\})=C_i$ and $c(T\cup\{v\})=C_j$.
\end{itemize}
Let $C_iC_j(uv)$ denote the collection of sets $T\in N_H(u)\cap N_H(v)$ for which $uTv$ is $C_iC_j$. Define $S(uv)$ analogously.
\end{define}
Note that $C_iC_j(uv)=C_jC_i(vu)$ for all distinct $u,v \in V(H)$.

\begin{define}\label{def3.2}
Let $D>0$ and let  $u,v\in V(H)$ be distinct. We say that $N_H(u)\cap N_H(v)$ is
\begin{itemize}
    \item \textbf{type $\bf S(D)$} if $|S(uv)|\geq D n^{k-2}$; 
    \item \textbf{type $\bf C_iC_j(D)$} if $i\neq j$ and $|C_iC_j(uv)|\geq D n^{k-2}$.
\end{itemize}    
\end{define} 
We remark that it may be the case that $N_H(u)\cap N_H(v)$ has more than one type.

\begin{define}\label{defgad}
    Let $e=\{e_1, \dots, e_k\}$ and $f=\{f_1, \dots, f_k\}$  be two  edges in $H$. A {\bf $\bf (k^2+k,e,f)$-gadget} $G$ is a subhypergraph of $H$ on $k^2+k$ vertices so that:
    \begin{itemize}
    \item   $V(G)$ is the disjoint union of $e$, $f$ and $T_1,\dots, T_k $ where $T_i\in N_H(e_i)\cap N_H(f_i)$ for each $i\in[k]$;
    \item $e,f \in E(G)$;
        \item  $e_iT_i, f_iT_i \in E(G)$ for all $i\in[k]$.
\end{itemize}
A $(k^2+k,e,f)$-gadget in which every $e_iT_if_i$ is $S$ will be called an \textbf{S-$\mathbf{(k^2+k,e,f)}$-gadget}.

A {\bf $\bf (3k,e,f)$-gadget} $G$ is a subhypergraph of $H$ on $3k$ vertices so that:
    \begin{itemize}
        \item $e_i = f_i$, for all $i\in\{3,\dots, k\}$;
        \item $V(G)$ is the disjoint union of $e$, $f_1$, $f_2$, $T_1$ and $T_2$, where $T_i\in N_H(e_i)\cap N_H(f_i)$ for each $i\in [2]$;
        \item $e,f\in E(G)$;
        \item $e_1T_1, f_1T_1, e_2T_2, f_2T_2  \in E(G)$.
    \end{itemize}
Given $t \in \{3k, k^2+k\}$,
    we say that a $(t,e,f)$-gadget $G$ is \textbf{good} if it contains two perfect matchings with different colour profiles (with respect to the $r$-colouring of $G$ induced by the $r$-colouring $c$ of $H$). 
\end{define}
Note that $e$ and $f$ are vertex-disjoint in a $(k^2+k,e,f)$-gadget but intersect in $k-2$ vertices in a $(3k,e,f)$-gadget; see Figure~\ref{fig:enter-label}.
\begin{figure}[H]
    \centering
    \includegraphics[scale=1.5]{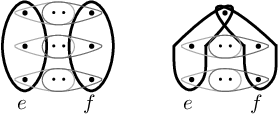}
    \caption{On the left, a $(12, e, f)$-gadget. On the right, a $(9,e,f)$-gadget.}
    \label{fig:enter-label}
\end{figure}

\subsection{Tools for the proof of Theorem~\ref{thmmain2}}
The following well-known result allows one to deduce a lower bound on $\delta_{\ell }(H)$ given a lower bound on $\delta_{\ell'}(H)$,  for any $\ell \leq \ell'$.
\begin{prop}\label{propdegree}
    Let $1\leq \ell \leq \ell' < k$ and $H$ be a $k$-uniform hypergraph on $n$ vertices. If $\delta_{\ell'}(H)\geq x{n-\ell' \choose k-\ell'}$ for some $0\leq x\leq 1$, then $\delta_{\ell}(H)\geq x{n-\ell \choose k-\ell}$. \qed
\end{prop}
The next result gives a sufficient condition for finding a  good $(3k, e, f)$-gadget in a $k$-uniform hypergraph of large minimum $2$-degree.
\begin{lem}\label{good3kgadget}
Let $k\geq 3$ and  $D:=3k$. Let $H$ be a $k$-uniform hypergraph on $n$ vertices with an $r$-colouring $c:E(H)\to \{C_1, \dots, C_r\}$. Suppose there exists $i\neq j\in [r]$ and distinct $v_1,v_2,v_3,v_4\in V(H)$ such that $N_H(v_1)\cap N_H(v_2)$ and $N_H(v_3)\cap N_H(v_4)$ are both type $C_iC_j(D)$. If
\[\delta_2(H) > \frac{1}{2} {n \choose k-2},\]
then there exists a good $(3k, e, f)$-gadget in $H$, for some $e,f\in E(H)$.
\end{lem}

\begin{proof}
By the minimum $2$-degree condition, there exists a set $X\subseteq V(H)$ of size $k-2$ such that $A = X\cup\{v_1,v_3\}$ and $B = X\cup\{v_2,v_4\}$ are both in $E(H)$. We show that we can construct a $(3k, A, B)$-gadget and afterwards we prove that it is good. 

Given that $N_H(v_1)\cap N_H(v_2)$ is type $C_iC_j(D)$, there are at least $3kn^{k-2}$ sets $T_{1,2}\in N_H(v_1)\cap N_H(v_2)$  such that $c(v_1T_{1,2}) = C_i$ and $c(v_2T_{1,2}) = C_j$. As $|A\cup B|=k+2 < 3k$, we may choose such a set
 $T_{1,2}$ so that it is also vertex-disjoint from $A\cup B$. Similarly, there is  a set $T_{3,4}\in N_H(v_3)\cap N_H(v_4)$ such that $c(v_3T_{3,4}) = C_i$, $c(v_4T_{3,4}) = C_j$ and $T_{3,4}$ is vertex-disjoint from $A$, $B$ and $T_{1,2}$. 

Then, define a gadget $G$ as follows:
\begin{itemize}
    \item $V(G)$ is the union of $A$, $B$, $T_{1,2}$ and $T_{3,4}$;
    \item $A$, $B$, $v_1T_{1,2}$, $v_2T_{1,2}$, $v_3T_{3,4}$ and $v_4T_{3,4}$ are in $E(G)$.
\end{itemize}
By definition, $G$ is a $(3k,A,B)$-gadget.

To prove that $G$ is good, we need to find two perfect matchings in $G$ with different colour profiles. Define $M_A := \{A, v_2T_{1,2}, v_4T_{3,4}\}$ and $M_B := \{B, v_1T_{1,2}, v_3T_{3,4}\}$. Both $M_A$ and $M_B$ are perfect matchings in $G$. While $M_A$ has at least two $C_j$-edges ($v_2T_{1,2}$ and $v_4T_{3,4}$), $M_B$ has at least two $C_i$-edges ($v_1T_{1,2}$ and $v_3T_{3,4}$). Thus, $M_A$ and $M_B$ have different colour profiles, as desired. 
\end{proof}
The next lemma ensures a hypergraph $H$ as in Theorem~\ref{thmmain2} contains a good gadget or a perfect matching of huge colour-bias.
\begin{lem}\label{gadgetsinH}
Let $2\leq\ell<k$ and $\eta>0$. There exists an $n_0\in \mathbb N$ such that the following holds for all $n\geq n_0$ with $n \in k \mathbb N$.
Let $H$ be a $k$-uniform hypergraph on $n$ vertices with an $r$-colouring $c:E(H)\to \{C_1, \dots, C_r\}$ and \[\delta_\ell (H) \geq (c_{k,\ell} + \eta){n \choose k-\ell}.\]
Suppose that $H$ does not have a perfect matching containing  at least $n/k - {r \choose 2}$ edges of the same colour. Then
\begin{itemize}
    \item there exists a good $(3k,e,f)$-gadget in $H$, for some $e,f\in E(H)$; or
    \item there exists a good $(k^2+k,e,f)$-gadget in $H$, for some $e,f\in E(H)$.
\end{itemize}
\end{lem}

\begin{proof}
Let $H$ and $c$ be as in the lemma and  suppose $n$ is sufficiently large. Let $D:=k^2+k\geq 3k$. Note that, given our minimum $\ell$-degree condition, Proposition~\ref{propdegree} implies that
\begin{align}\label{mindeg}
\delta_1(H) \geq (c_{k,\ell}+\eta){n-1 \choose k-1}> 
\left (\frac{1}{2} +\frac{\eta}{2} \right )\binom{n}{k-1} \text{ \ and \ }
\delta_2(H) \geq (c_{k,\ell}+\eta){n-2 \choose k-2} >
\frac{1}{2} \binom{n}{k-2}.
\end{align}
Here the inequalities follow as $c_{k,\ell}\geq 1/2$ by (\ref{mlower}).

As $n$ is sufficiently large, and by definition of $c_{k,\ell}$,  the minimum $\ell$-degree condition ensures a perfect matching $M$ in $H$.

Let $L := {r\choose 2}+1$. 
By the hypothesis of the lemma, $M$ does not contain $n/k - {r\choose 2}$ edges of the same colour; so there exist distinct edges $e_1,\dots, e_L, f_1, \dots ,f_L \in M$
such that $c(e_i)\neq c(f_i)$ for each $i\in [L]$. 

Given any distinct $x,y \in V(H)$, (\ref{mindeg}) implies that $|N_H(x)\cap N_H(y)|\geq \eta \binom{n}{k-1}$. In particular, this means that $N_H(x)\cap N_H(y)$ is of type $S(D)$ or of type $C_iC_j(D)$ for some distinct $i,j\in [r]$.

Suppose there exists $i\neq j\in [r]$ and distinct $x,y,z,w\in V(H)$ such that $N_H(x)\cap N_H(y)$ and $N_H(z)\cap N_H(w)$ are both type $C_iC_j(D)$.
Then by Lemma~\ref{good3kgadget}, there exists a good $(3k, e, f)$-gadget in $H$, for some $e,f\in E(H)$.

So we may assume no such  $i\neq j\in [r]$ and $x,y,z,w\in V(H)$ exist. In particular, for each of the 
${r\choose 2}=L-1$ choices for $i\neq j\in [r]$, 
there is at most one pair $(e_s,f_s)$ such that 
there exists a $u \in e_s$ and $v \in f_s$ so that either  
$N_H(u)\cap N_H(v)$ or $N_H(v)\cap N_H(u)$ is type $C_iC_j(D)$.
Thus, the following claim holds.

\begin{claim}
    There is a pair $(e_s,f_s)$ such that for each
    $u \in e_s$ and $v \in e_s$ we have that 
    $N_H(u)\cap N_H(v)$ is type $S(D)$.
\end{claim}
Let $e_s=\{u_1,\dots, u_k\}$ and $f_s=\{v_1,\dots, v_k\}$.
 For each $i\in[k]$, we  choose a set $T_i$ so that
\begin{enumerate}
    \item[(i)] $T_i \in S(u_iv_i)$;
    \item[(ii)] $T_1,\dots, T_k, e_s,f_s$ are all vertex-disjoint.
\end{enumerate}
Note we can guarantee (ii) since $|S(u_iv_i)|\geq Dn^{k-2}=(k^2+k)n^{k-2}$ for each $i \in [k]$.

We construct a $(k^2+k, e_s, f_s)$-gadget $G$ as follows:\begin{itemize}
    \item $V(G)$ is the union of $e_s$, $f_s$, $T_1, \dots ,T_k$;
    \item $e_s$ and $f_s$ are edges in $G$; 
    \item $u_iT_i$, $v_iT_i$ are edges in $G$ for all $i \in [k]$.
\end{itemize}
By definition,  $G$ is an $S$-$(k^2+k,e_s,f_s)$-gadget with $c(e_s) \neq c(f_s)$. This implies that
$G$ is a good $(k^2+k,e_s,f_s)$-gadget.
Indeed, $M_e:=\{e_s, v_1T_1, \dots , v_kT_k\}$ and
$M_f:=\{f_s, u_1T_1, \dots , u_kT_k\}$ are perfect matchings in $G$ with different colour profiles.
\end{proof}

\section{Proof of Theorem~\ref{thmmain2}}
Let $H$ be a sufficiently large $n$-vertex $k$-uniform hypergraph as in the statement of the theorem. Let $c: E(H)\to \{C_1, \dots, C_r\}$ be an $r$-colouring of $E(H)$.
If $H$ contains a perfect matching with at least $n/k - {r \choose 2}$ edges of the same colour, then we are done. 

So, suppose no perfect matching in $H$ contains at least $n/k - {r \choose 2}$ edges of the same colour. By Lemma~\ref{gadgetsinH}, we can find either a good $(3k,e,f)$-gadget  or a good $(k^2+k,e,f)$-gadget in $H$. Call this gadget $G_1$. 

Next consider $H_1:=H\setminus V(G_1)$. Clearly $\delta _\ell (H_1)\geq (c_{k,\ell}+\eta /2)\binom{n}{k-\ell}$. 
Suppose  $H_1$ contains a perfect  matching $M_1$ with 
at least $|H_1|/k - {r \choose 2}$ edges of the same colour.
Thus, by taking any perfect matching in $G_1$ and adding it to $M_1$, we obtain a perfect matching in $H$ containing at least 
$|H_1|/k - {r \choose 2}\geq n/k  -|G_1|/k -{r \choose 2}\geq n/k -k-1-{r \choose 2}$ edges of the same colour, as desired.

Hence, we may assume $H_1$ does not contain such a perfect  matching $M_1$.  By Lemma~\ref{gadgetsinH}, we can find either a good $(3k,e,f)$-gadget  or a good $(k^2+k,e,f)$-gadget in $H_1$. Call this gadget $G_2$ and set $H_2:= H_1\setminus V(G_2)$.

Repeating this argument, we either obtain a perfect matching in $H$ of significant colour-bias, or a collection of
$t := \frac{\eta n}{4 k^k (k^2+k)}$ vertex-disjoint gadgets
$G_1,\dots, G_t$ where, given any $i \in [t]$, $G_i$ is either 
 a good $(3k,e,f)$-gadget  or a good $(k^2+k,e,f)$-gadget in $H$.
In particular, note that each gadget we select has size at most $k^2+k$, and if one removes $t(k^2+k)$ vertices
from $H$ one still has that $\delta _{\ell} (H) \geq (1/2+\eta)\binom{n}{k-\ell}-t(k^2+k)n^{k-\ell-1}\geq (1/2+\eta/2)\binom{n}{k-\ell}$. Thus, we can indeed repeatedly apply  Lemma~\ref{gadgetsinH} to obtain these gadgets $G_1,\dots, G_t$. 

Set $\mathcal{G}:=\{ G_1,\dots, G_t\}$.
For each colour $C_i$, consider the set $\mathcal{G}_i$ of all the gadgets in $\mathcal{G}$ that contain two perfect matchings with different colour profiles with respect to the colour $C_i$. Clearly there  exists some $j \in [r]$
such that
$\mathcal{G}_j$ contains at least $t/r$ gadgets. 

For each gadget $G_i$ in $\mathcal{G}_j$ consider the perfect matching $M_i$ in $G_i$ with the largest possible number of edges coloured $C_j$; let $M'_i$ be the perfect matching  in $G_i$ with the fewest possible  edges coloured $C_j$. So $M_i$ has at least one more $C_j$-edge than $M'_i$.

Let $M^+$ denote the union of all these $M_i$ and let $M^-$ denote the union of all these $M'_i$. So 
$M^+$ contains at least $t/r=\frac{\eta n}{4 r k^k (k^2+k)}$
more $C_j$-edges than $M^-$.

Let $V(\mathcal{G}_j)$ denote the set of vertices in $H$ that lie in one of the gadgets in $\mathcal{G}_j$.
Note that $\delta_{\ell} (H\setminus V(\mathcal{G}_j))\geq
(c_{k,\ell}+\eta/2)\binom{n}{k-\ell}$ so there exists a perfect matching $M$ in $H\setminus V(\mathcal{G}_j)$. Thus, $M\cup M^+$ and $M\cup M^-$ are both perfect matchings in $H$.

If $M\cup M^-$ contains at least $\frac{n}{rk}+\frac{\eta n}{8r(r-1)k^k(k^2+k)} $ edges of the same colour then the theorem holds. Thus, we may assume this is not the case. 
This immediately implies the following claim.
    
\begin{claim}
    For every $i\in[r]$, the number of $C_i$-edges  in $M\cup M^-$ is at least $\frac{n}{rk} - \frac{\eta n}{8rk^k(k^2+k)} $.
\end{claim}
In particular, $M\cup M^-$ contains at least $\frac{n}{rk} - \frac{\eta n}{8rk^k(k^2+k)} $ $C_j$-edges. Since there are 
at least 
$\frac{\eta n}{4 r k^k (k^2+k)}$
more $C_j$-edges in $M^+$ than in $M^-$, we obtain that 
$M\cup M^+$ contains at least $\frac{n}{rk} +\frac{\eta n}{8rk^k(k^2+k)} $ $C_j$-edges, as desired.\qed

\section{Concluding Remarks}\label{sec:conc}
In this paper we have determined the minimum $\ell$-degree threshold for forcing a colour-bias perfect matching in a $k$-uniform hypergraph for all $2\leq \ell <k$. The only remaining open case of the problem is the minimum \emph{vertex} degree version.

A result of H\`an,  Person and  Schacht~\cite{hps} yields that $m_1(3,n)=(5/9+o(1))\binom{n-1}{2}$. The following example shows that the corresponding colour-bias problem has a significantly higher minimum vertex degree threshold.

\begin{example}
    Given any $n \in 6\mathbb N$, there exists an $n$-vertex $3$-uniform hypergraph $H$ with 
    $$\delta_1(H) \geq \frac{3}{4} {n -1 \choose 2}$$
    and a $2$-colouring of $E(H)$ so that every perfect matching in $H$ has precisely $n/6$ edges in each colour.
\end{example}
\begin{proof}
    Define $H$ so that (i) $V(H)$ is the disjoint union of two vertex classes $A$ and $B$, both of size $n/2$; (ii) $E(H)$ consists of all those $3$-uniform edges containing at least one vertex from each of $A$ and $B$. Thus, 
    $$\delta_1(H)=\binom{n/2}{2}+\frac{n}{2}\left (\frac{n}{2}-1 \right)\geq \frac{3}{4}{n -1 \choose 2}.$$
    Colour each edge containing $2$ vertices from $A$ red; each edge containing $2$ vertices from $B$ blue. It is easy to see that every perfect matching in $H$ uses the same number of red and blue edges.
\end{proof}
We suspect that this example is  extremal  for the minimum vertex degree problem in $3$-uniform hypergraphs.
\begin{question}
    Given any $\eta>0$ does there exists a $\gamma>0$ so that the following holds for all sufficiently large $n\in 3\mathbb N$? Suppose that $H$ is an $n$-vertex $3$-uniform hypergraph with
    $$\delta_1(H)\geq \left(\frac{3}{4} +\eta\right) {n -1 \choose 2}.$$ Then given any $2$-colouring of $E(H)$ there is a perfect matching in $H$ with at least $n/6+\gamma n$ edges of the same colour.
\end{question}
    
\section*{Acknowledgment}
Part of the research in this paper was carried out during a visit by the first author to the University of Birmingham in July 2023. The authors are grateful to the BRIDGE strategic alliance between the University of Birmingham and the University of Illinois at Urbana-Champaign, which partially funded this visit. 

{\noindent \bf Data availability statement.}
There are no additional data beyond that contained within the main manuscript.

\end{document}